\documentclass[12pt]{amsart}

\usepackage{amsopn}
\usepackage{amssymb}
\usepackage{amscd}
\usepackage[all]{xy}

\newtheorem{theorem}{Theorem}[section]
\newtheorem{lemma}[theorem]{Lemma}
\newtheorem{proposition}[theorem]{Proposition}
\newtheorem{corollary}[theorem]{Corollary}
\newtheorem*{Theorem}{Theorem}
\newtheorem*{Proposition}{Proposition}
\newtheorem*{Corollary}{Corollary}
\newtheorem*{Notation}{Notation}

\theoremstyle{definition}
\newtheorem{definition}[theorem]{Definition}
\newtheorem{example}[theorem]{Example}

\theoremstyle{remark}
\newtheorem{remark}[theorem]{Remark}

\newtheorem{question}[theorem]{Question}
\newtheorem*{Assertion}{Assertion}

\numberwithin{equation}{section}

\begin{document}

\title [Properness, Cauchy-indivisibility and the Weil completion]{Properness, Cauchy-indivisibility and the Weil completion of a group of isometries}

\author[A. Manoussos]{A. Manoussos}
\address{Fakult\"{a}t f\"{u}r Mathematik, SFB 701, Universit\"{a}t Bielefeld, Postfach 100131, D-33501 Bielefeld, Germany}
\email{amanouss@math.uni-bielefeld.de}
\thanks{During this research the first author was fully supported by SFB 701 ``Spektrale Strukturen und
Topologische Methoden in der Mathematik" at the University of Bielefeld, Germany. He is grateful for its generosity and hospitality. }

\author[P. Strantzalos]{P. Strantzalos}
\address{Department of Mathematics, University of Athens, Panepistimioupolis, GR-157 84, Athens, Greece}
\email{pstrantz@math.uoa.gr}
\thanks{}



\begin{abstract}
In this paper we introduce a new class of metric actions on separable (not necessarily connected) metric spaces called ``Cauchy-indivisible" actions. This new class
coincides with that of proper actions on locally compact metric spaces and, as examples show, it may be different in general. The concept of ``Cauchy-indivisibility"
follows a more general research direction proposal in which we investigate the impact of basic notions in substantial results, like the impact of local compactness
and connectivity in the theory of proper transformation groups. In order to provide some basic theory for this new class of actions we embed a ``Cauchy-indivisible"
action of a group of isometries of a separable metric space in a proper action of a semigroup in the completion of the underlying space. We show that, in case this
subgroup is a group, the original group has a ``Weil completion" and vice versa. Finally, in order to establish further connections between ``Cauchy-indivisible"
actions on separable metric spaces and proper actions on locally compact metric spaces we investigate the relation between ``Borel sections" for ``Cauchy-indivisible"
actions and ``fundamental sets" for proper actions. Some open questions are added.
\end{abstract}

\maketitle

\section{Introduction}
We begin with a brief summary of a research proposal, that will be published in more details elsewhere.

The History of Mathematics indicates that during the $19^{th}$ century the mathematical creativity lived a quantitative and -mainly- a qualitative explosion, which is
initiated and developed by a critical renovation of notions and procedures that have being initiated the foregoing centuries. Our period can be characterized as one
where we understand deeper and broader the achievements of the $19^{th}$ century, as it was the case a few decades before the beginning of the $19^{th}$ century when
the fundamental notions began to be clarified. These remarks led to a research direction which -roughly speaking- amounts to the following general requirement:
Walking towards the origins of the $19^{th}$ century, try to critically understand the impact of functionally crucial assumptions in substantial results, preferably
within composed frames, therefore frames which work broadly in Mathematics and their applications. Such a frame is, certainly, that of Topological Transformation
Groups with applications, for instance, in Topology, Geometry and Physics.

The above research direction led to \cite{MS2}, where the impact of ``property Z" (: every compact subset of the underlying space is contained in a compact and
connected one) in the main result of \cite{abels1} is investigated. The main result of \cite{MS2} is that, in almost all cases, under the restrictive ``property Z"
was hidden that, without assuming this property, there always exists a maximal zero-dimensional compactification of the underlying space which does the same job as in
\cite{abels1}, that may be different from the (absolutely maximal) end-point compactification considered in \cite{abels1}, and coincides with it, among others, if
``property Z" is satisfied.

In the paper at hand, within the same research direction and having in mind the fruitful theory of proper transformation groups on locally compact and connected
spaces, we propose an analogous class of actions, not necessarily proper, without assuming local compactness and connectedness of the underlying spaces. So, we
introduce a new, rather natural, class of metric actions on separable (not necessarily connected) metric spaces called ``Cauchy-indivisible". Note that the isometric
actions constitute nowadays an important part of the theory of proper actions and that the group of isometries of a locally compact and connected metric space acts
properly on it.

As the following definition shows, ``Cauchy-indivisible" actions are characterized by an ``\textit{isotropic}" behavior of divergent nets of the acting group with
respect to the basic metric notion of a ``Cauchy sequence". Recall that ``$z_{i}\to\infty$ in $Z$" means that the net $\{ z_{i}\}$ does not have any convergent subnet
in the space $Z$.

\begin{definition} \label{def21}
Let $(G,X)$ be a continuous action of a topological group $G$ on a metric space $X$. The action is said to be \textit{Cauchy-indivisible} if the following holds: If
$\{g_{i}\}$ is a net in $G$ such that $g_{i}\to\infty$ in $G$ and $\{ g_{i}x\}$ is a Cauchy net in $X$ for some $x\in X$ then $\{ g_{i}x\}$ is a Cauchy net for every
$x\in X$.
\end{definition}

It turns out that a Cauchy-indivisible action on a locally compact or complete metric space is proper and vice versa (cf. $\S3$), and that in general the two notions
may differ (cf. $\S4$). In both cases the underlying space is not assumed to be connected. The omission of this assumption in the locally compact case, as well the
omission of the local compactness in the main part of the paper at hand is an advantage coming from the fact that Cauchy-indivisibility essentially reflects the
\textit{global} behavior of selfmaps of $X$ comparing to \textit{local} properties or the \textit{connectedness} of the underlying space. So we can generalize the
framework of proper actions and go beyond, in accordance to our research direction proposal, provided that this new framework leads (a) to interesting results in the
non locally compact case and (b) enables a better understanding of proper actions on locally compact spaces.

Concerning requirement (b) we note that in Theorem \ref{th33} we give an answer to the open question of characterizing proper actions on non-connected locally compact
metric spaces and in Theorem \ref{sectfund} we establish an interconnection between Borel sections (which occur in Cauchy-indivisible actions on separable spaces, cf.
Proposition \ref{prop71}) and fundamental sets that characterize proper isometric actions. Recall that a \textit{section} of an action $(G,X)$ is a subset of $X$
which contains only one point from each orbit. A \textit {Borel section} is a section that is Borel subset of $X$ (e.g. useful in measure theory).

\begin{Theorem}[Theorem \ref{sectfund}]
Let $G$ be a locally compact group which acts properly on a locally compact space $X$, and suppose that the orbit space $G\backslash X$ is paracompact. Let $S$ be a
section for the action $(G,X)$. Then
\begin{itemize}
\item[(i)] For every open neighborhood $U$ of $S$ we can construct a closed fundamental set $F_c$ and an open fundamental set $F_o$ such that $F_c\subset F_o\subset U$.
\item[(ii)] If, in addition, $(X,d)$ is a separable metric space, in which case the action $(G,X)$ is Cauchy-indivisible, then there exists a Borel section $S_B$, which
is also a fundamental set, such that $S_B\subset F_c\subset F_o\subset U$.
\end{itemize}
\end{Theorem}

Note that $S_B$ in (ii) of the above theorem, is a ``minimal" fundamental set, because of its construction, and as such may lead to applications.

The new notion ``like properness" seams to be suitable for structure theorems, as our first results indicate. Concerning requirement (a) above, in $\S 5$, which is
the main part of the paper at hand we consider a separable metric space $(X,d)$ such that the natural evaluation action of the group of isometries $Iso(X)$ on $X$ is
Cauchy-indivisible. Let $\widehat{X}$ denote the completion of $(X,d)$ and let $E$ be the \textit{Ellis semigroup} of the lifted group $\widehat{Iso(X)}$ in
$C(\widehat{X},\widehat{X})$, i.e. the pointwise closure of $\widehat{Iso(X)}$ in $C(\widehat{X},\widehat{X})$. Let

\begin{align*}
& \hspace{-5mm} H =\{ h\in C(\widehat{X},\widehat{X})\,|\,\mbox{there exists a sequence}\,\,\{ g_{n}\}\subset Iso(X)\\
& \hspace{11mm}\mbox{with}\,\, g_{n}\to\infty\,\,\mbox{in}\,\, Iso(X)\,\,\mbox{and}\,\, \widehat{g_{n}}\to h \,\,\mbox{in}\,\, C(\widehat{X},\widehat{X}) \},\\
& \hspace{-5mm} X_l=\{ hx\,|\, h\in H,\, x\in X\},\\
&  \hspace{-5mm} X_p=\{ hx\,|\, h\in H\cap Iso(\widehat{X}),\, x\in X \}.
\end{align*}

With this notation and the previous mentioned assumptions, among other results we show that

\begin{Theorem}[Theorem \ref{th512}]
The set $X\cup X_{p}$ is the maximal subset of $X\cup X_{l}$ that contains $X$ such that the map
\[
\omega :E\times (X\cup X_{p})\to (X\cup X_{p})\times \widehat{X}
\]
with $\omega (f,y)=(y,fy)$, $f\in E$ and $y\in X\cup X_{p}$ is proper.
\end{Theorem}

The interest on this theorem lies on the fact that an action $(G,X)$ is proper if the map $G\times X\to X \times X$ defined by $(g,x)\mapsto (x,gx)$ is proper, cf.
\cite[Definition 1, p.250]{bour1}.

We remind that a topological group has a \textit{Weil completion} with respect to the uniformity of pointwise convergence if it can be embedded densely in a complete
group with respect to its left uniform structure.

\begin{Proposition}[Proposition \ref{prop516}]
The following are equivalent:
\begin{enumerate}
\item[(i)] The map
\[
\omega :E\times (X\cup X_{l})\to (X\cup X_{l})\times \widehat{X}
\]
is proper.

\item[(ii)] $E$ is a group (precisely a closed subgroup of $Iso(\widehat{X})$).

\item[(iii)] $Iso(X)$ has a Weil completion.
\end{enumerate}
\end{Proposition}

\begin{Corollary}[Corollary \ref{cor518}]
If $E$ is a group the action $(Iso(X),X)$ is embedded densely in the proper action $(E,X\cup X_{l})$ such that the following equivariant diagram commutes
\[
\xymatrix{
(Iso(X),X)       \ar[d] \ar[r]   &X \ar[d] \\
(E,X\cup X_{l})  \ar[r]        &\widehat{X}}
\]
where $X\to X\cup X_{l}$ is the inclusion map and the map $Iso(X)\to E$ is defined by $g\mapsto \widehat{g}$ for every $g\in Iso(X)$. By the word ``densely" we mean
that $X$ is dense in $X\cup X_{l}$ and $\widehat{Iso(X)}$ is dense in $E$.
\end{Corollary}

The above result may lead to further structure theorems, cf. Question \ref{que519}.

\begin{Proposition}[Proposition \ref{prop71}]
If the Ellis semigroup $E$ is a group then the action $(E,X\cup X_{l})$ has a Borel section.
\end{Proposition}

As Theorem \ref{sectfund}(ii) mentioned above indicates, the notion of a Borel section, which according to the above result is a feature of the Cauchy-indivisible
actions on separable metric spaces, is remarkably related to that of a fundamental set in the locally compact case and may be, similarly, used for structural
theorems. So, it is interesting to ask whether the existing Borel section for the action $(E,X\cup X_l)$ can be reduced to a Borel section for the initial action
$(Iso(X),X)$, cf. Question \ref{que73}.

In order to indicate or to exclude possible directions for further investigation concerning Cauchy-indivisible actions, we study various examples, cf. Examples
\ref{ex513}, \ref{ex515}, \ref{ex520} and \ref{ex76}. Among them, an example that may be of independent interest, is the following (cf. $\S6$): Consider the action
$(Iso(Iso(\mathbb{Z})),Iso(\mathbb{Z}))$, where $\mathbb{Z}$ is the discrete space of the integers, and with suitable metrics on the  acting group
$Iso(Iso(\mathbb{Z}))$ and the underlying space $Iso(\mathbb{Z})$. We show that this action is proper and Cauchy-indivisible while the Ellis semigroup is not a group
and $Iso(Iso(\mathbb{Z}))$ has no Weil completion.

\section{Basic notions and notation}
For what follows, and in addition to the notation established in the introduction, $(X,d)$ will denote a metric space with metric $d$ and $Iso(X)$ will denote its
group of (surjective) isometries of $X$ endowed with the topology of pointwise convergence. With this topology $Iso(X)$ is a topological group \cite[Ch. X, \S 3.5
Corollary]{bour2}. Let $(\widehat{X},\widehat{d})$ stands for the completion of $(X,d)$. For a Cauchy sequence $\{ x_{n}\}$ in $X$ let $[x_{n}]\in\widehat{X}$ denote
the limit point of $\{ x_{n}\}$ in $\widehat{X}$. We denote by $\widehat{g}$ and $\widehat{Iso(X)}$ the lift of $g\in Iso(X)$ and the lift of the group $Iso(X)$
respectively in $C(\widehat{X},\widehat{X})$, the space of the continuous selfmaps of $\widehat{X}$ (which is considered with the topology of pointwise convergence).

A continuous action of a topological group $G$ on a topological space $X$ is a continuous map $G\times X\to X$ with $(g,x)\mapsto gx$, $g\in G, x\in X$ such that
$(e,g)\mapsto x$, for every $x\in X$ where $e$ denotes the unit element of $G$, and $(h,(g,x))\mapsto (hg)x$ for every $h,g\in G$ and $x\in X$. When the action map is
known we will denote the action simply by $(G,X)$. Let $U\subset X$, then $GU:=\{gx \,|\, g\in G, x\in U\}$. Especially, if $U=\{x\}$ then the set $Gx:=G\{x\}$ is
called the orbit of $x\in X$ under $G$. The subgroup $G_x:=\{g\in G \,|\, gx=x\}$ of $G$ is called the isotropy group of $x\in X$. The natural evaluation action of
$Iso(X)$ on $X$ is the map $Iso(X)\times X\to X$ with $(g,x)\mapsto g(x)$, $g\in Iso(X), x\in X$ and is denoted by $(Iso(X),X)$. If we endow $Iso(X)$ with the
topology of pointwise convergence this action is always continuous. As usual, $S(x,\varepsilon)$ will denote the open ball centered at $x$ with radius $\varepsilon
>0$.

\begin{definition}\label{proper}
A continuous action $(G,X)$ is (equivalently to the Bourbaki-definition) proper if $J(x)=\emptyset$, for every $x\in X$, where
\[
\begin{split}
J(x)=\{ &y\in X\,|\, \mbox{there exist nets}\,\, \{x_{i}\}\,\, \mbox{in}\,\, X\,\, \mbox{and}\,\, \{g_{i}\}\,\, \mbox{in}\,\, G\\
        &\mbox{with}\,\, g_{i}\to\infty, \lim x_{i}=x\,\, \mbox{and}\,\, \lim g_{i}x_{i}=y\}
\end{split}
\]
denotes the extended (prolongational) limit set of $x\in X$.
\end{definition}

It is easily seen that in the special case of actions by isometries $J(x)=L(x)$ holds for every $x\in X$, where
\[
\begin{split}
L(x)=\{ &y\in X\,|\, \mbox{there exists a net}\, \{g_{i}\}\, \mbox{in}\,\, G\\
        &\mbox{with}\,\, g_{i}\to\infty\,\, \mbox{and}\,\, \lim g_{i}x=y\},
\end{split}
\]
denote the limit set of $x\in X$ under the action of $G$ on $X$. Hence an action by isometries $(G,X)$ is proper if and only if $L(x)=\emptyset$ for every $x\in X$.

\section{Cauchy-indivisibility and proper actions on locally compact metric spaces}
In this section we show that for group actions on locally compact metric spaces the notions of properness and Cauchy-indivisibility coincide. We start with the
following easily proved observation.

\begin{lemma} \label{lem31}
Let $(X,d)$ be a locally compact metric space and $\{ g_{i}\}\subset Iso(X)$ be a net such that $\{ g_{i}x\}$ is a Cauchy net for some $x\in X$. Then there exists a
point $y\in X$ such that $g_ix\to y$.
\end{lemma}

\begin{proposition} \label{prop32}
Let $(X,d)$ be a locally compact metric space. The action $(Iso(X),X)$ is proper if and only if it is Cauchy-indivisible.
\end{proposition}
\begin{proof}
Assume that  $(Iso(X),X)$ is Cauchy-indivisible. We will show that the limit sets $L(x)$ are empty for every $x\in X$. Assume the contrary, that is there exist a net
$\{ g_{i}\}$ in $Iso(X)$ and $x,y\in X$ such that $g_{i}\to\infty$ and $g_{i}x\to y$. We will show that $g_{i}\to h$ for some $h\in Iso(X)$, which is a contradiction
to the assumption $g_{i}\to\infty$.  Since $(Iso(X),X)$ is Cauchy-indivisible then $\{ g_{i}x\}$ is a Cauchy net, for every $x\in X$.  Therefore,  by the previous
lemma, there is a map $h:X\to X$ defined by $h(x):=y$ such that $g_i\to h$ pointwise on $X$ and $h$ preserves the metric $d$. Observe that $g_i^{-1}y\to x$ since
$d(g_i^{-1}y,x)=d(y,g_ix)$. Applying Cauchy-indivisibility for the action $(Iso(X),X)$ and the previous lemma again, we conclude that there exists a map $f:X\to X$
such that $g_i^{-1}\to f$ pointwise on $X$ and $f$ preserves the metric $d$. Obviously $f$ is the inverse map of $h$, hence $h\in Iso(X)$.

The converse implication follows easily in a similar way.
\end{proof}

If $X$ is locally compact and $G$ acts properly on $X$ (hence $G$ is a locally compact group), it is well known, see e.g. \cite{koszul}, that there exists a
$G$-invariant compatible metric on $X$. Compatible means that this metric induces the topology of $X$. Hence, the previous proposition states the following result
that characterizes the properness of actions on locally compact metric spaces independently of the connectedness of the underlying space.

\begin{theorem} \label{th33}
Let $(X,d)$ be a locally compact metric space. An action $(G,X)$ is proper if and only if it is Cauchy-indivisible.
\end{theorem}

\begin{remark}
The previous theorem also holds, and can be similarly proved, if we replace the full group of isometries of $X$ by a closed subgroup of it or if we replace the local
compactness of $X$ by completeness.
\end{remark}

\section{Cauchy-indivisibility vs properness}
In this section we provide examples showing that Cauchy-indivisi\-bi\-li\-ty and properness are distinct notions for isometric actions on separable and non locally
compact metric spaces. We also provide some criteria for the coexistence of Cauchy-indivisible and proper actions on the basis of the dynamical behavior of the
lifting of the action $(Iso(X),X)$ in the completion of the underlying space.

\begin{remark} \label{rem41}
The example in \S 6 shows that the two notions may coexist also in the case when $X$ is neither locally compact nor complete.
\end{remark}

The following example shows that the action $(Iso(X),X)$ can be proper and not Cauchy-indivisible.

\begin{example} \label{ex42}
Let $X$ be the set $\mathbb{Q}$ of the rational numbers endowed with the Euclidean metric. It is easy to see that the action $(Iso(X),X)$ is proper. Take a sequence
of rational numbers $\{ q_{n}\}$ such that $q_{n}\to a$, where $a$ is an irrational. Let $\{ g_{n}\}\subset Iso(X)$ with $g_{n}x:=(-1)^{n}x+q_{n}$ for every $x\in X$,
then $g_{n}\to\infty$ in $Iso(X)$. Since $g_{n}0=q_{n}$ for every $n\in \mathbb{N}$, the sequence $\{ g_{n}0\}$ is Cauchy. But for $x\neq 0$ the sequence $\{
g_{n}x\}$ has two limit points in $\mathbb{R}$ hence cannot be a Cauchy sequence.
\end{example}

The following example shows that the action $(Iso(X),X)$ can be Cauchy-indivisible and not proper.

\begin{example} \label{ex43}
Let $X$ be the set $\mathbb{Q}+\sqrt{2}\,\mathbb{N}$  endowed with the Euclidean metric. Its group of isometries is $\mathbb{Q}$ acting by translations (reflections
are excluded because of the addend $\sqrt{2}\,\mathbb{N}$). Therefore, $(Iso(X),X)$ is Cauchy-indivisible. However, the action $(Iso(X),X)$ is not proper. To see that
take a sequence of rational numbers $\{ q_{n}\}$ such that $q_{n}\to \sqrt{2}$. Let $\{ g_{n}\}\subset Iso(X)$ with $g_{n}x:=x+q_{n}$. Observe that
$g_n^{-1}\sqrt{2}\to 0\notin X$. Therefore $g_n\to \infty$ in $Iso(X)$. Since $g_n\sqrt{2}\to 2\sqrt{2}\in X$ the limit set $L(\sqrt{2})$ is not empty, so the action
$(Iso(X),X)$ is not proper.
\end{example}

Motivated by these examples we give necessary and sufficient conditions for a Cauchy-indivisible action $(Iso(X),X)$ to be proper and vice versa:

\begin{proposition} \label{prop44}
Let $Iso(X)$ be Cauchy-indivisible. The following are equivalent:
\begin{itemize}
\item[(i)]  The action $(Iso(X),X)$ is proper.
\item[(ii)] If $h$ is in the pointwise closure of $\widehat{Iso(X)}$ in $C(\widehat{X},\widehat{X})$ then either $h(X)\subset X$ or $h(X)\subset \widehat{X}\setminus X$.
\end{itemize}
\end{proposition}
\begin{proof}
Assume that the action $(Iso(X),X)$ is proper and $h$ is in the pointwise closure of $\widehat{Iso(X)}$ in $C(\widehat{X},\widehat{X})$. Then there is a net $\{
\widehat{g_{i}}\}$ in $\widehat{Iso(X)}$ such that $\widehat{g_{i}}\to h$ pointwise in $\widehat{X}$. If $h(X)\cap X\neq\emptyset$ then there is some $x\in X$ such
that $\widehat{g_{i}}x\to hx\in X$. Since the action $(Iso(X),X)$ is proper the net $\{ g_{i}\}$ has a convergent subnet in $Iso(X)$. Then it is easy to see that
$h\in\widehat{Iso(X)}$, hence $h(X)\subset X$.

Assume now that condition (ii) holds. We will show that the limit sets $L(x)$ are empty for every $x\in X$, hence the action $(Iso(X),X)$ is proper. We will proceed
by contradiction. Assume that there exist $x,y\in X$ and a net $\{ g_{i}\}$ in $Iso(X)$ with $g_{i}x\to y$ and $g_{i}\to\infty$ in $Iso(X)$. Since $\{ g_{i}x\}$ is a
Cauchy net in $X$ and $Iso(X)$ is Cauchy-indivisible then $\{ g_{i}x\}$ is a Cauchy net for every $x\in X$, hence $\{ g_{i}x\}$ converges in $\widehat{X}$ for every
$x\in X$. So, we can define a map $h:X\to\widehat{X}$ by letting $hx:=\lim \widehat{g_{i}}x$. It is easy to see that $h$ preserves the metric $\widehat{d}$ on $X$.
Thus, if $w\in\widehat{X}$ and $\{ x_{n}\}\subset X$ is a sequence in $X$ such that $x_{n}\to w$ in $\widehat{X}$ then $\{ hx_{n}\}$ is a Cauchy sequence in $X$,
hence it converges to a point in $\widehat{X}$ which is independent of the choice of the sequence $\{ x_{n}\}$. Then, by \cite[Ch. I, \S 8.5 Theorem 1]{bour1}, the
map $h:X\to \widehat{X}$ has a unique continuous extension on $\widehat{X}$. It is easy to see that $\widehat{g_{i}}\to h$ pointwise on $\widehat{X}$, thus $h$ is in
the pointwise closure of $\widehat{Iso(X)}$ in $C(\widehat{X},\widehat{X})$. Since $g_{i}x\to y$ then $hx=y$ where $x,y\in X$. So using our hypothesis $h(X)\subset
X$. Since $g_i$ preserves the metric $d$ then $g_{i}^{-1}y\to x$. Using the same arguments as before we have that $h\in\widehat{Iso(X)}$ hence the net $\{ g_{i}\}$
converges in $Iso(X)$, a contradiction to the assumption $g_{i}\to\infty$ in $Iso(X)$.
\end{proof}

\begin{proposition} \label{prop45}
Assume that $(Iso(X),X)$ is a proper action. The following are equivalent:
\begin{itemize}
\item[(i)]  $Iso(X)$ is Cauchy-indivisible.
\item[(ii)] Let $\{ g_{i}\}\subset Iso(X)$ a net with $g_{i}\to\infty$ and $\{ g_{i}x\}$ be a Cauchy net for some $x\in X$. If $y\in X$ then the net $\{ g_{i}y \}$ can not
have more than one limit point in the completion $\widehat{X}$ of $X$.
\end{itemize}
\end{proposition}
\begin{proof}
The direction from (i) to (ii) is trivial. If the converse implication does not hold, then there is a Cauchy net $\{ g_{i}x\}$ such that there is $y\in X$, an
$\varepsilon >0$ and subnets $\{ g_{i_{k}}y\}$, $\{ g_{i_{l}}y\}$ of $\{ g_{i}y \}$ such that $d(g_{i_{k}}y,g_{i_{l}}y)\geq\varepsilon$ for every index $k,l$. Since
$\{ g_{i}x\}$ is a Cauchy net in $X$ then we may assume that $d(g_{i_{k}}x,g_{i_{l}}x)\to 0$. Hence, $d(g_{i_{k}}^{-1}g_{i_{l}}x,x)\to 0$. We can define a new net $\{
h_{i,j}\}\subset Iso(X)$ by letting $h_{i,j}:=g_{j}^{-1}g_{i}$ for every pair of indices $(i,j)$, with direction defined by $(i_{1},j_{1})\leq (i_{2},j_{2})$ if and
only if $i_{1}\leq i_{2}$ and $j_{1}\leq j_{2}$. Therefore, $h_{i_{k},i_{l}}x\to x$. Since $(Iso(X),X)$ is proper there is a subnet $\{ h_{i_{k_{m}},i_{l_{m}}} \}$
and some $g\in Iso(X)$ such that $h_{i_{k_{m}},i_{l_{m}}}\to g$. Hence $\{ h_{i_{k_{m}},i_{l_{m}}}y \}$ is a Cauchy net in $X$, therefore for every $\varepsilon '
>0$ there exists an index $m_{0}$ such that
\[
d(g_{i_{k_{m}}}^{-1}g_{i_{l_{m}}}y,g_{i_{k_{n}}}^{-1}g_{i_{l_{n}}}y)<\varepsilon ' \,\,\, \mbox{for every}\,\, m,n\geq m_{0}.
\]
By taking $m=n\geq m_{0}$ it is easy to see that $\{ g_{i_{l_{m}}}y\}$ is a Cauchy net and if we follow the same procedure we can also show that $\{ g_{i_{k_{m}}}y\}$
is also a Cauchy net. Since $d(g_{i_{k_{m}}}y,g_{i_{l_{m}}}y)\geq\varepsilon$ for every index $m$ the net $\{ g_{i}y \}$ has two limit point in the completion
$\widehat{X}$ of $X$, a contradiction to our hypothesis.
\end{proof}

\section{Cauchy-indivisible isometric actions on separable metric spaces}
\textit{In this section $(X,d)$ will denote a separable metric space such that the action $(Iso(X),X)$ is Cauchy-indivisible}. Firstly, we show the adequacy of
sequences in the definition of Cauchy-indivisibility.

\begin{proposition} \label{51}
In the definition of Cauchy-indivisibility for isometric actions nets can be replaced by sequences.
\end{proposition}
\begin{proof}
Assume that if $\{ g_{n}\}$ is a sequence in $Iso(X)$ such that $g_{n}\to\infty$ and $\{ g_{n}x\}$ is a Cauchy sequence in $X$ for some $x\in X$ then $\{ g_{n}x\}$ is
a Cauchy sequence for every $x\in X$. Let $\{ f_{i}\}$ be a net in $Iso(X)$ such that $f_{i}\to\infty$ and $\{ f_{i}x\}$ is a Cauchy net in $X$ for some $x\in X$. We
will show that $\{ f_{i}x\}$ is a Cauchy net in $X$ for every $x\in X$. We argue by contradiction. Suppose that there exists $y\in X$ such that $\{ f_{i}y\}$ is not a
Cauchy net. Hence, there is an $\varepsilon >0$ and subnets $\{ f_{i_{k}} \}$, $\{ f_{i_{l}} \}$ such that $d(f_{i_{k}}y,f_{i_{l}}y)\geq\varepsilon$ for every $k,l$.
Since $\{ f_{i}x\}$ is a Cauchy net in $X$ there is a point $z\in \widehat{X}$ such that $\widehat{f_{i}}x\to z$. Hence, the subnets $\{ \widehat{f_{i_{k}}}x\}$, $\{
\widehat{f_{i_{l}}}x\}$ also converges to $z$. So we may find sequences $\{ \widehat{f_{i_{k_{n}}}}x\}$ and $\{ \widehat{f_{i_{l_{n}}}}x\}$ such that
$\widehat{f_{i_{k_{n}}}}x\to z$ and $\widehat{f_{i_{l_{n}}}}x\to z$. Therefore, $\{ f_{i_{k_{n}}}x \}$ and $\{ f_{i_{l_{n}}}x \}$ are Cauchy sequences in $X$ and
$d(f_{i_{k_{n}}}y,f_{i_{l_{n}}}y )\geq\varepsilon$ for every $n\in\mathbb{N}$. Let $\{ h_{n}\}\subset Iso(X)$ with
\[
h_{4n-3}=f_{i_{k_{2n-1}}},\,\,h_{4n-2}=f_{i_{l_{2n-1}}},\,\,h_{4n-1}=f_{i_{l_{2n}}}\,\,\mbox{and}\,\, h_{4n}=f_{i_{k_{2n}}},
\]
$n=1,2\ldots$. It is easy to see that $\widehat{h_{n}}x\to z$, hence $\{ h_{n}x\}$ is a Cauchy sequence in $X$. Moreover, $\{ h_{n}y\}$ is not a Cauchy sequence in
$X$ since $d(f_{i_{k_{n}}}y,f_{i_{l_{n}}}y )\geq\varepsilon$ for every $n\in\mathbb{N}$ and for the same reason $h_{n}\to\infty$ in $Iso(X)$, which is a contradiction
to our hypothesis.
\end{proof}

\begin{definition} \label{def52}
Fix a dense sequence $D=\{x_{i}\}\subset X$ in $\widehat{X}$. Since the metric $\frac{\widehat{d}}{1+\widehat{d}}$ is an equivalent metric to $\widehat{d}$ on $X$
(also gives the same groups of isometries on $X$ and $\widehat{X}$ and the same Cauchy sequences) we may assume that $\widehat{d}$ is bounded by $1$. We define
$\delta: Iso(\widehat{X})\times Iso(\widehat{X})\to \mathbb{R}^{+}$ by
\[
\delta(f,g)=\sum_{i=1}^{\infty} \frac{1}{2^{i}}\,\, \widehat{d}(fx_{i},gx_{i})
\]
for every $f,g\in Iso(\widehat{X})$. It is easy to see that $\delta$ is a left-invariant metric on $Iso(\widehat{X})$.
\end{definition}

\begin{proposition} \label{prop53}
The uniformity of pointwise convergence, the left uniformity and the uniformity induced by $\delta$ on $Iso(\widehat{X})$ and $Iso(X)$ coincide, independently of
Cauchy-indivisibility.
\end{proposition}
\begin{proof}
The proof is similar to the proof of Lemma 2.11 in \cite{hjorth}.
\end{proof}

\begin{proposition} \label{prop54}
The pointwise closures of $Iso(X)$ in $C(X,X)$ and of $Iso(\widehat{X})$ in $C\widehat{X},\widehat{X})$ endowed with the metric $\delta$ are separable metric spaces.
\end{proposition}
\begin{proof}
It follows easily using the same arguments as in the proof of Lemma 2.11 in \cite{hjorth} and \cite[Ch. X, \S 3 Exercise 6 (b), p. 327]{bour2}.
\end{proof}

The following lemma will be used often in the sequel.

\begin{lemma} \label{lem55}
Let $\{ g_{n}\}$ be a sequence in $Iso(X)$ such that $\{ g_{n}x\}$ is a Cauchy sequence in $X$ for some $x\in X$ and $g_n\to\infty$. Then
\begin{enumerate}
\item[(i)] $\{ g_{n}x_{n}\}$ is a Cauchy sequence for every Cauchy sequence $\{x_{n}\}$ in $X$.
\item[(ii)] If $\{x_{k}\}$ is Cauchy sequence in $X$ then $\widehat{g_{n}}[x_{k}]\to [g_{k}x_{k}]$ in $\widehat{X}$.
\end{enumerate}
\end{lemma}
\begin{proof}
(i) The proof follows by the triangle inequality and the fact that $\{g_{n}x_{n_{0}}\}$ is a Cauchy sequence, for suitable $n_0\in\mathbb{N}$.

(ii) By item (i), $\{g_{k}x_{k}\}$ is a Cauchy sequence in $X$, hence $[g_{k}x_{k}]\in\widehat{X}$. The rest of the proof is similar to that of item (i).
\end{proof}

\begin{corollary} \label{cor56}
If $\{ g_{n}\}$ is a sequence in $Iso(X)$ such that $g_n\to\infty$ and $\{ g_{n}x\}$ is a Cauchy sequence in $X$ for some $x\in X$, then $\{ \widehat{g_{n}}\}$
converges pointwise on $\widehat{X}$ to some $h\in C(\widehat{X},\widehat{X})$ which preserves the metric $\widehat{d}$. In addition, if $\{g_{n}^{-1}y\}$ is a Cauchy
sequence for some $y\in X$, then $\{\widehat{g_{n}}\}$ converges pointwise on $\widehat{X}$ to some $h\in Iso(\widehat{X})$.
\end{corollary}
\begin{proof}
The proof is an immediate consequence of Lemma \ref{lem55} (ii) if we set $h:\widehat{X}\to\widehat{X}$ with $h[x_{k}]:=[g_{k}x_{k}]$ for every
$[x_{k}]\in\widehat{X}$.
\end{proof}

Corollary \ref{cor56} enables the following equivalent expressions of the corresponding sets defined in the introduction:

\begin{Notation} \label{not1}
\begin{align*}
&  \hspace{-5mm} H=\{ h\in C(\widehat{X},\widehat{X})\,|\,\mbox{there exists a sequence}\,\,\{ g_{n}\}\subset Iso(X)\\
&  \hspace{12mm}\mbox{with}\,\, g_{n}\to\infty\,\,\mbox{in}\,\,Iso(X),\,\,\{ g_{n}x\}\,\,\mbox{is a Cauchy sequence}\\
&  \hspace{12mm}\mbox{for some}\,\, x\in X\,\,\mbox{and}\,\,\widehat{g_{n}}\to h \,\,\mbox{in}\,\, C(\widehat{X},\widehat{X})\}.
\end{align*}
\begin{align*}
& X_{l} =\{ y\in\widehat{X}\,|\,\mbox{there exists a sequence}\,\,\{ g_{n}\}\subset Iso(X)\,\, \mbox{with}\\
& \hspace{12mm}g_{n}\to\infty\,\,\mbox{in}\,\, Iso(X),\,\,\mbox{such that}\,\, \{ g_{n}x\}\,\,\mbox{is a Cauchy}\\
& \hspace{12mm}\mbox{sequence for some}\,\, x\in X\,\, \mbox{and}\,\,y=[g_{k}x]\},\\
& \mbox{denotes the set of the limit points of the}\,\,\mbox{action}\,\,(Iso(X),X)\,\,\mbox{in}\,\,\widehat{X}.
\end{align*}
\begin{align*}
&\hspace{-2mm} X_{p}=\{ y\in\widehat{X}\,|\,\mbox{there exists a sequence}\,\,\{ g_{n}\}\subset Iso(X)\,\,\mbox{with}\\
& \hspace{12mm}g_{n}\to\infty\,\,\mbox{in}\,\, Iso(X),\,\,\mbox{such that}\,\, \{ g_{n}x\}\,\,\mbox{and}\,\,\{ g_{n}^{-1}x\}\,\,\mbox{are}\\
& \hspace{12mm}\mbox{Cauchy sequences for some}\,\, x\in X\,\, \mbox{and}\,\, y=[g_{k}x]\},\\
& \mbox{denotes the set of the special}\,\,\mbox{limit points of}\,\,(Iso(X),X)\,\,\mbox{in}\,\,\widehat{X}.
\end{align*}
\end{Notation}

\begin{proposition} \label{prop57}
If $\{g_{n}\}$ is a sequence in $Iso(X)$ such that $g_{n}\to f$ on $X$ for some $f$ in $C(X,X)$, then $\widehat{g_{n}}\to \widehat{f}$ on $\widehat{X}$ and
$\widehat{f}\in Iso(\widehat{X})$.
\end{proposition}
\begin{proof}
If $\{g_{n}\}$ has a convergent subsequence $\{g_{n_{k}}\}$ to some point $g\in Iso(X)$ then $f=g$ on $X$. We will show that $\widehat{g_{n}}\to \widehat{g}$
pointwise on $\widehat{X}$. Take some $y\in\widehat{X}$ and an $\varepsilon >0$. Then there exists some $x\in X$ such that
$\widehat{d}(x,y)<\frac{1}{3}\,\varepsilon$. Since $g_{n}x\to gx$, there is a positive integer $n_{0}$ such that $d(g_{n}x,gx)<\frac{1}{3}\,\varepsilon$ for every
$n\geq n_{0}$. Hence, $\widehat{d}(\widehat{g_{n}}y,\widehat{g}y)<\varepsilon$, for every $n\geq n_{0}$.

Assume, now, that $g_{n}\to\infty$ in $Iso(X)$. Since $g_{n}\to f$ on $X$, then $\{g_{n}x\}$ is a Cauchy sequence, for every $x\in X$. Hence, by Corollary
\ref{cor56}, $\{ \widehat{g_{n}}\}$ converges pointwise to some $h\in C(\widehat{X},\widehat{X})$. Since $\widehat{f}x=fx=hx$ for every $x\in X$, then $\widehat{f}=h$
on $\widehat{X}$. Note that if $[x_{k}]\in \widehat{X}$ then $\widehat{f}[x_{k}]=h[x_{k}]:=[g_{k}x_{k}]$. Next, we show that $\widehat{f}$ is surjective. Let
$[y_{k}]\in\widehat{X}$. Since $g_{n}\to f$ on $X$ and $fx\in X$ for every $x\in X$, then $g_{n}^{-1}fx\to x$. Hence, by Lemma \ref{lem55} (i) and the
Cauchy-indivisibility of the action $(Iso(X),X)$ we have that $[g^{-1}_{k}y_{k}]\in\widehat{X}$. Therefore, by Lemma \ref{lem55} (ii),
$\widehat{d}(\widehat{f}[g^{-1}_{k}y_{k}],[y_{k}])=\lim_{n} \widehat{d}(\widehat{g_{n}}[g^{-1}_{k}y_{k}],[y_{k}])=\widehat{d}([g_{k}g^{-1}_{k}y_{k}],[y_{k}])
=\widehat{d}([y_{k}],[y_{k}])=0$. This finishes the proof.
\end{proof}

With the notation established in the introduction, we have

\begin{proposition} \label{prop58}
The set $E$ is
\begin{enumerate}
\item[(i)] the union $\widehat{Iso(X)}\cup H$,
\item[(ii)] complete with respect to the uniformity of pointwise convergence on $\widehat{X}$, and
\item[(iii)] a semigroup, the Ellis semigroup of $\widehat{Iso(X)}$ in $C(\widehat{X},\widehat{X})$, i.e. the pointwise closure of
 $\widehat{Iso(X)}$ in $C(\widehat{X},\widehat{X})$.
\end{enumerate}
\end{proposition}
\begin{proof}
(i) Take  a sequence $\{\widehat{g_{n}}\}$ in $\widehat{Iso(X)}$ such that $\widehat{g_{n}}\to h$ for some $h\in C(\widehat{X},\widehat{X})$. If $\{g_{n}\}$ has a
convergent subsequence  to some $g\in Iso(X)$ then, by Proposition \ref{prop57}, $h=\widehat{g}\in\widehat{Iso(X)}$. Let $g_{n}\to\infty$ in $Iso(X)$ and take some
$x\in X$. Since $\widehat{g_{n}}x\to hx$, then $\{g_{n}x\}$ is a Cauchy sequence in $X$ therefore $h\in H$.  Items (ii) and (iii) follow from Lemmata 2.10 and 2.11 in
\cite{hjorth} by noticing that a sequence $\{ g_n\}$ in $Iso(X)$ is Cauchy with respect to the left uniformity of $Iso(X)$ if and only if $\{ g_nx\}$ is Cauchy in $X$
for every $x\in X$.
\end{proof}

\begin{remark} \label{rem59}
As the example in \S 6 shows, the Ellis semigroup $E$ is not in general a group. However
\end{remark}

\begin{proposition} \label{prop510}
The Ellis semigroup $E$ is a group if and only if $X_{l}=X_{p}$.
\end{proposition}
\begin{proof}
Assume that $E$ is a group and let $y\in X_{l}$. Hence, there is a sequence $\{ g_{n}\}\subset Iso(X)$ with $g_{n}\to\infty$ in $Iso(X)$ and a map $h\in
C(\widehat{X},\widehat{X})$ such that $\widehat{g_{n}}\to h$ pointwise on $\widehat{X}$ and $y=hx$ for some $x\in X$. Since $E$ is a group then $h$ has an inverse
$h^{-1}$. Thus $\widehat{g_{n}}^{-1}\to h^{-1}$. The last implies that $\{ g_{n}^{-1}x\}$ is a Cauchy sequence in $X$, therefore $y\in X_{p}$.

To show the converse implication, assume that $X_{l}=X_{p}$ and take some $h\in E$. By Proposition \ref{prop58} (i), $h\in\widehat{Iso(X)}\cup H$. So, if
$h\in\widehat{Iso(X)}$ obviously it has an inverse. Assume that $h\in H$. Hence, there is a sequence $\{ g_{n}\}\subset Iso(X)$ with $g_{n}\to\infty$ in $Iso(X)$ such
that $\widehat{g_{n}}\to h$ pointwise on $\widehat{X}$. So $[g_{n}x]\in X_{l}$ for every $x\in X$. But $X_{l}=X_{p}$, hence $\{ g_{n}^{-1}x\}$ is a Cauchy sequence.
Applying Corollary \ref{cor56}, $h\in Iso(\widehat{X})$ so it has an inverse in $E$.
\end{proof}

\begin{lemma} \label{lem511}
The set $X\cup X_{l}$ is $E$-invariant.
\end{lemma}
\begin{proof}
It is easy to verify that $X$ and $X_{l}$ are $\widehat{Iso(X)}$-invariant. We will show that are also $H$-invariant. Let $h\in H$ and $x\in X$. By the definition of
$H$ there is some sequence $\{g_{n}\}$ in $Iso(X)$ such that $g_{n}\to\infty$ in $Iso(X)$ and $\widehat{g_{n}}\to h$ pointwise on $\widehat{X}$. If $[f_{n}x]\in
X_{l}$, for some sequence $\{f_{n}\}\subset Iso(X)$ and $x\in X$ then, by Corollary \ref{cor56}, $h[f_{n}x]=[g_{n}f_{n}x]$. If the sequence $\{g_{n}f_{n}\}$ has a
convergent subsequence in $Iso(X)$ then the Cauchy sequence $\{g_{n}f_{n}x\}$ has a convergent subsequence in $X$, so it converges in $X$. So
$h[f_{n}x]=[g_{n}f_{n}x]\in X$. Otherwise $g_{n}f_{n}\to\infty$ and $h[f_{n}x]=[g_{n}f_{n}x]\in X_{l}$.
\end{proof}

\begin{theorem} \label{th512}
The set $X\cup X_{p}$ is the maximal subset of $X\cup X_{l}$ that contains $X$ such that the map
\[
\omega :E\times (X\cup X_{p})\to (X\cup X_{p})\times \widehat{X}
\]
with $\omega (f,y)=(y,fy)$, $f\in E$ and $y\in X\cup X_{p}$ is proper.
\end{theorem}
\begin{proof}
We firstly show that the map $\omega :E\times (X\cup X_{p})\to (X\cup X_{p})\times \widehat{X}$ is proper. Since the evaluation map $E\times (X\cup
X_{p})\to\widehat{X}$ is isometric and action-like, according to \S 2, it suffices to show that the limit sets $L(x)$ are empty for every $x\in X\cup X_{p}$. Let $\{
f_{n}\}$ be a sequence in $E$ such that $f_{n}y\to z$ for some $y\in X\cup X_{p}$ and $z:=[z_k]\in \widehat{X}$.

\noindent {\it Case I.} Assume that $y\in X$. If $\{ f_{n}\}$ has a subsequence $\{ f_{n_{k}}\}$ in $\widehat{Iso(X)}$ then either the restriction of $\{ f_{n_{k}}\}$
on $X$ has a convergent subsequence in $Iso(X)$ hence, by Proposition \ref{prop57}, the sequence $\{ f_{n_{k}}\}$ converges pointwise to some point of
$\widehat{Iso(X)}\subset E$, or $f_n\to\infty$ in $Iso(X)$. In this case, since $\{f_{n}y\}$ is a Cauchy sequence in $X$, the sequence $\{ f_{n}\}$ converges
pointwise to some point of $H\subset E$ by Corollary \ref{cor56}.

Assume, now, that $\{ f_{n}\}$ is in $H$ and consider the dense sequence $D=\{x_{i}\}$ in $X$ which we used to define the metric $\delta$; cf. Definition \ref{def52}.
So, there is a sequence $\{x_{i_{n}}\}$ in $D$ such that $x_{i_{n}}\to y$. By the definition of $H$ and Proposition \ref{prop53}, there is a sequence $\{g_{n}\}$ in
$Iso(X)$  such that
\begin{eqnarray} \label{eq51}
\delta(\widehat{g_{n}}, f_{n})<\frac{1}{i_{n}2^{i_n}}.
\end{eqnarray}
Hence, using the form of the metric $\delta$, we conclude that
\[
\widehat{d}(\widehat{g_{n}}x_{i_{n}},f_{n}x_{i_{n}})<\frac{1}{i_{n}}.
\]
Moreover,
\[
\begin{split}
\widehat{d}(\widehat{g_{n}}y,z)&\leq \widehat{d}(\widehat{g_{n}}x_{i_{n}},f_{n}x_{i_{n}}) + \widehat{d}(f_{n}x_{i_{n}},f_{n}y) + \widehat{d}(f_{n}y,z)\\
                               &=\widehat{d}(\widehat{g_{n}}x_{i_{n}},f_{n}x_{i_{n}}) + \widehat{d}(x_{i_{n}},y) + \widehat{d}(f_{n}y,z).
\end{split}
\]
Therefore, $g_{n}y\to z$. Arguing as in the beginning of the proof, $\{ g_{n}\}$ has a convergent subsequence to a point of $E$, hence by equation \ref{eq51}, the
same holds for the sequence $\{ f_{n}\}$.

\medskip \noindent {\it Case II.} Assume that $y\in X_{p}$. Hence, there exist a sequence $\{ p_{k}\}\subset Iso(X)$ with $p_{k}\to\infty $ in $Iso(X)$, an isometry
$h_{1}\in Iso(\widehat{X})$ such that $\widehat{p_{k}}\to h_{1}$ pointwise on $\widehat{X}$ and $h_{1}x:=[p_{k}x]=y$ for some $x\in X$. If $\{ f_{n}\}$ has a
subsequence $\{ f_{n_{k}}\}$ in $\widehat{Iso(X)}$ then either the restriction of $\{ f_{n_{k}}\}$ on $X$ has a convergent subsequence in $Iso(X)$ hence, by
Proposition \ref{prop57}, the sequence $\{ f_{n_{k}}\}$ converges pointwise to some point of $\widehat{Iso(X)}\subset E$, or $f_{n}\to\infty$ in $Iso(X)$. If the
later holds, then we will show that there is a Cauchy sequence of the form $\{f_{n_{i}}p_{k_{i}}x\}$ in $X$ for some subsequences $\{f_{n_{i}}\}$ and $\{p_{k_{i}}\}$
of  $\{f_{n}\}$ and $\{p_{k}\}$ respectively (the problem is that we do not know if $\{f_{n}x\}$ or $\{f_{n}p_{n}x\}$ is a Cauchy sequence in $X$ for some $x\in X$).

Let $i$ be a positive integer. Since $f_{n}[p_{k}x]\to z$ and $z:=[z_k]\in \widehat{X}$, there is a positive integer $n_{0}$ that depends only on $i$ such that
\[
\widehat{d}(f_{n}[p_{k}x],[z_k])<\frac{1}{i}
\]
for every $n\geq n_{0}(i)$. Therefore
\[
\lim_{k} d(f_{n}p_{k}x,z_k):=\widehat{d}(f_{n}[p_{k}x],[z_k])<\frac{1}{i}
\]
for every $n\geq n_{0}(i)$. Hence, using induction, we may find strictly increasing sequences of positive integers $\{n_{i}\}$ and $\{k_{i}\}$ such that
\begin{eqnarray} \label{eq52}
d(f_{n_{i}}p_{k_{i}}x,z_{k_i})<\frac{1}{i}
\end{eqnarray}
for every positive integer $i$. Since $\{ z_{k_i}\}$ is a Cauchy sequence then by equation \ref{eq52}, $\{f_{n_{i}}p_{k_{i}}x\}$ is a Cauchy sequence in $X$.

Now, either $\{f_{n_{i}}p_{k_{i}}\}$ has a convergent subsequence in $Iso(X)$ (without loss of generality and for the economy of the proof we may assume that
$\{f_{n_{i}}p_{k_{i}}\}$ converges in $Iso(X)$) or $f_{n_{i}}p_{k_{i}}\to\infty$ in $Iso(X)$. In both cases, by Corollary \ref{cor56} and Proposition \ref{prop57},
there is $h_{2}\in C(\widehat{X},\widehat{X})$ such that $\widehat{f_{n_{i}}}\widehat{p_{k_{i}}}=\widehat{f_{n_{i}}p_{k_{i}}}\to h_{2}$ pointwise on $\widehat{X}$. We
will show that $\widehat{f_{n_{i}}}\to h_{2}h_{1}^{-1}$ pointwise on $\widehat{X}$. Take $w\in\widehat{X}$. Since $h_{1}\in Iso(\widehat{X})$, there is some $u\in
\widehat{X}$ such that $h_{1}(u)=w$. Hence
\[
\begin{split}
\widehat{d}(f_{n_{i}}w,h_{2}h_{1}^{-1}w)&=\widehat{d}(f_{n_{i}}h_{1}u,h_{2}u)\\
                                        &\leq \widehat{d}(f_{n_{i}}h_{1}u,\widehat{f_{n_{i}}p_{k_{i}}}u) + \widehat{d}(\widehat{f_{n_{i}}p_{k_{i}}}u,h_{2}u)\\
                                        &=\widehat{d}(h_{1}u,\widehat{p_{k_{i}}}u) + \widehat{d}(\widehat{f_{n_{i}}p_{k_{i}}}u,h_{2}u)
\end{split}
\]
which converges to $0$, since $\widehat{p_{k_{i}}}\to h_{1}$ and $\widehat{f_{n_{i}}p_{k_{i}}}\to h_{2}$ pointwise on $\widehat{X}$. Hence $\{f_{n}x\}$ is a Cauchy
sequence for every $x\in X$. Since we assumed that $f_{n}\to\infty$ in $Iso(X)$ then, by Corollary \ref{cor56}, $\{f_{n}\}$ converges pointwise on $\widehat{X}$ to
$h_{2}h_{1}^{-1}\in E$.

To finish the proof of the second case assume that $\{f_{n}\}$ is in $H$. Then arguing as in the first case we can show that $\{ f_{n}\}$ has a convergent subsequence
to a point of $E$.

Next, we show that if $Y$ is a subset of $X\cup X_{l}$ that contains $X$ such that the map
\[
\omega :E\times Y\to Y\times \widehat{X}
\]
is proper then $Y\subset X\cup X_{p}$. To see that take a point $[g_{k}x]\in Y\setminus X$. This means that $\{ g_{k}\}$ is a sequence in $Iso(X)$ such that
$g_{k}\to\infty$ in $Iso(X)$ and $\{ g_{k}x\}$ is a Cauchy sequence in $X$. By Lemma \ref{lem55} (ii), $\widehat{g_{n}}x\to [g_{k}x]$ and, by Corollary \ref{cor56},
$\{ \widehat{g_{n}}\}$ converges pointwise on $\widehat{X}$ to some $h\in C(\widehat{X},\widehat{X})$. Note that $x\in X\subset Y$. Since $\widehat{g_{n}}x\to
[g_{k}x]$ and $\widehat{g_n}$ preserves the metric $\widehat{d}$ then $\widehat{g_{n}}^{-1}[g_{k}x]\to x$, Hence, by the properness of $\omega$, we may assume that
$\{\widehat{g_{n}}^{-1}\}$ has a subsequence $\{\widehat{g_{n_{k}}}^{-1}\}$ that converges pointwise to some $f\in E$. This makes $h$ a surjection, hence $h\in
Iso(\widehat{X})$. Therefore, $[g_{k}x]\in X_{p}$, so $Y\setminus X \subset X_{p}$.
\end{proof}

Note that, as the following example shows, it may happen that $X_{p}=X_{l}\neq\emptyset$, $X\cup X_{p}\neq\widehat{X}$ and \textit{the set $X\cup X_{p}$ is not the
maximal subset of $\widehat{X}$ such that the action $(E,X\cup X_{p})$ is proper}.

\begin{example} \label{ex513}
Take
\[
X:=\{ (x,y)\in\mathbb{R}\,|\, x\in \mathbb{Q}+\sqrt{2}\,\mathbb{N},y>0\,\},
\]
endowed with the Euclidean metric. Its group of isometries is the additive group of the rational numbers acting by horizontal translations. Therefore, $(Iso(X),X)$ is
Cauchy-indivisible. Obviously $\widehat{X}$ is the closed upper half plane, $X_{p}=X_{l}\neq\emptyset$, $X\cup X_{p}$ is the open upper half plane and $E$ is the
additive group of the real numbers acting by horizontal translations on $\widehat{X}$. Hence $E$ acts properly on $\widehat{X}$.
\end{example}

\begin{remark} \label{rem514}
The sets $X_p$ and $X_l$ constructed in Theorem \ref{th512} are \textit{optimal} in the sense that if one may think to replace the sets $X_p$ and $X_l$ with the
following more general sets
\begin{align*}
&X^*_{l} =\{ y\in\widehat{X}\,|\,\mbox{there exist a sequence}\,\,\{ g_{n}\}\subset Iso(X)\,\,\mbox{and some}\,\,x\in X\\
&\hspace{12mm}\mbox{such that}\,\, g_{n}\to\infty\,\,\mbox{in}\,\, Iso(X), \{ g_{n}x\}\,\,\mbox{is a Cauchy sequence and}\\
& \hspace{12mm}y=[g_{k}x_k],\,\,\mbox{for some}\,\, [x_k]\in\widehat{X}\},\,\,\mbox{and}\\
& X^*_{p}=\{ y\in\widehat{X}\,|\,\mbox{there exist a sequence}\,\,\{ g_{n}\}\subset Iso(X)\,\,\mbox{and some}\,\,x\in X\\
& \hspace{12mm}\mbox{such that}\,\,g_{n}\to\infty\,\,\mbox{in}\,\, Iso(X), \{ g_{n}x\}\,\,\mbox{and}\,\,\{ g_{n}^{-1}x\}\,\,\mbox{are Cauchy}\\
& \hspace{12mm}\mbox{sequences and}\,\, y=[g_{k}x_k],\,\,\mbox{for some}\,\, [x_k]\in\widehat{X}\}
\end{align*}
and ask if the set $X\cup X^*_{p}$ is the maximal subset of the completion $\widehat{X}$ such that the map $\omega^* :E\times (X\cup X^*_{p})\to (X\cup X^*_{p})\times
\widehat{X}$ with $\omega^* (f,y)=(y,fy)$, $f\in E$ and $y\in X\cup X^*_{p}$ is proper \textit{this is not true in general}. This follows from the following assertion
and Example \ref{ex515}, which shows that there is a metric space $X$ such that $(Iso(X),X)$ is Cauchy-indivisible,  $X_p\neq\emptyset$ and the map $\omega^*$ as
above is not proper.

\begin{Assertion}
If $X^*_p\neq\emptyset$ (equivalently $X_{p}\neq\emptyset$) then $X^*_p=\widehat{X}$.
\end{Assertion}
\begin{proof} Let $y=[x_k]$ in $\widehat{X}$. By assumption, there exist a sequence $\{ g_n\}\subset Iso(X)$ and a point $x\in X$ such that $g_n\to\infty$ in $Iso(X)$ and the
sequences $\{ g_{n}x\}$, $\{ g_{n}^{-1}x\}$ are Cauchy. By Lemma \ref{lem55}, $\{ g_nx_n\}$ is a Cauchy sequence in $X$ and  $g^{-1}_k[g_kx_k]\to
[g^{-1}_{k}g_kx_k]=[x_k]=y$. Hence, $y\in X^*_p$.
\end{proof}
\end{remark}

\begin{example} \label{ex515}
The example is a combination of Example \ref{ex43} and of a 3-dimensional variation of the ``river metric" \cite[Example 4.1.6]{engel}. Let
\[
X=\{ (x,y,z)\,|\,\, x\in\mathbb{Q}+\sqrt{2}\,\mathbb{N},\,\,\, y\in\mathbb{Q}+\sqrt{2}\,\mathbb{N},\,\,\, z>0\}.
\]
For every pair of points $w_1=(x_1,y_1,z_1)$, $w_2=(x_2,y_2,z_2)\in X$ define
\[
d(w_1,w_2):=\left \{
\begin{array}{ll} |y_1 - y_2|+|z_1 - z_2|, & \mbox{if}\; x_1=x_2\\
          |y_1|+|y_2|+|x_1-x_2|+|z_1-z_2|, &\mbox{if}\; x_1\neq x_2.
\end{array} \right.
\]
We can easily verify than $d$ is a metric on $X$. The group of isometries $Iso(X,d)$ consists of all the maps $g:X\to X$ of the form
\[
g(x,y,z)=(x+p,y+q,z),\,\,\, p,q\in\mathbb{Q}.
\]
The action $(Iso(X),X)$ is Cauchy-indivisible since $X$ does not contain the $xy$-plane (the last coordinate of the points of $X$ is positive). Then
\[
X_p=\{ (x,y,z)\,|\,\, x\in\mathbb{Q}+\sqrt{2}\,\mathbb{N},\,\,\, y\in\mathbb{R},\,\,\, z>0\}.
\]
To see that take $x\in\mathbb{Q}+\sqrt{2}\,\mathbb{N}$, $y\in\mathbb{R}$ and $z>0$ and choose $k\in\mathbb{N}$ such that $y-\sqrt{2}k\notin\mathbb{Q}$. Let $\{
q_{n}\}$ be a sequence of rational numbers such that $q_n\to y-\sqrt{2}k$. Hence, if we let $\{ g_n\}\subset Iso(X)$ with $g_{n}(x,y,z):=(x,y+q_n,z)$ then
$g_{n}(x,\sqrt{2}k,z)=(x,q_n+\sqrt{2}k,z)\to (x,y,z)$. Hence $(x,y,z)\in X_p$. Observe that
\[
\widehat{X}=\{ (x,y,z)\,|\,\, x\in\mathbb{Q}+\sqrt{2}\,\mathbb{N},\,\,\, y\in\mathbb{R},\,\,\, z\geq 0\},
\]
and $E$ consists of all the maps $g:\widehat{X}\to \widehat{X}$ with
\[
g(x,y,z)=(x+p,y+r,z),\,\,\, p\in\mathbb{Q},\,r\in\mathbb{R}.
\]
However, the map $\widehat{\omega} :E\times \widehat{X}\to \widehat{X} \times \widehat{X}$ with $\widehat{\omega} (f,w)=(w,fw)$, $f\in E$ and $w\in \widehat{X}$ is
not proper since if we take a sequence of rational numbers $\{ p_{n}\}$ such that $p_{n}\to \sqrt{2}$ and let $\{ g_{n}\}\subset E$ with $g_{n}(x,y,z)=(x+p_n,y,z)$
then $g_n(x,0,0)\to (x+\sqrt{2},0,0)$ for each $x\in\mathbb{Q}+\sqrt{2}\,\mathbb{N}$. The sequence $\{ g_n\}$ diverges in $E$ since for instance the distance of the
points $ g_n(\sqrt{2},\sqrt{2},1)=(q_n+\sqrt{2},\sqrt{2},1)$ from any point of $X$ is eventually at least $\sqrt{2}$. Hence the limit set $L((x,0,0))$ is not empty.
\end{example}

A question that arises naturally from Theorem \ref{th512} is if the action of the Ellis semigroup $E$ on $X\cup X_{l}$ is proper. Surprisingly, as the following
proposition shows, this is equivalent to the existence of a Weil completion (with respect to the uniformity of pointwise convergence) for the group $Iso(X)$. Before
we give the statement let us recall a few things about the Weil completion of $Iso(X)$, defined in the introduction. The uniformity of pointwise convergence on $X$
coincides with the left uniformity of $Iso(X)$ (cf. \cite[Ch. III, \S 3.1 and Ch. X, \S 3 Exercise 19 (a), p. 332]{bour1}) and $Iso(X)$ has Weil completion with
respect to this uniformity if  the left and the right uniformities coincide; cf. \cite[Ch. III, \S 3.4 and \S 3 Exercise 3, p. 306]{bour1}. Note that the left
completion of $Iso(X)$ does not depend on the choice of a left-invariant metric on $Iso(X)$; cf. Lemma 2.9 in \cite{hjorth}.

\begin{proposition} \label{prop516}
The following are equivalent:
\begin{enumerate}
\item[(i)] The map
\[
\omega :E\times (X\cup X_{l})\to (X\cup X_{l})\times \widehat{X}
\]
is proper.

\item[(ii)] $E$ is a group (precisely a closed subgroup of $Iso(\widehat{X})$).

\item[(iii)] $Iso(X)$ has a Weil completion with respect to the uniformity of pointwise convergence (in this case $E$ is the Weil completion of $Iso(X)$).
\end{enumerate}
\end{proposition}
\begin{proof}
We show that item (i) implies item (ii) and vice versa. Suppose that $\omega$ is  proper. Take some $h\in E$. Since $E$ is a semigroup, cf. Proposition \ref{prop58}
(iii), we have only to show that $h$ has an inverse in $E$. If $h=\widehat{g}\in\widehat{Iso(X)}$ for some $g\in Iso(X)$, then $\widehat{g^{-1}}$ is the inverse of
$h$ in $\widehat{Iso(X)}\subset E$. If $h\in H$ there is a sequence $\{g_{n}\}$ in $Iso(X)$ such that $g_{n}\to\infty$ in $Iso(X)$ and $\widehat{g_{n}}\to h$
pointwise on $\widehat{X}$. Hence, if $x\in X$ then $\widehat{g_{n}}x\to hx$. Since $\widehat{g_n}$ preserves the metric $\widehat{d}$ then
$\widehat{g_{n}}^{-1}hx=\widehat{g_{n}^{-1}}hx\to x$. By Lemma \ref{lem511}, $hx\in X\cup X_{l}$, hence, by the properness of $\omega$, $\{ \widehat{g_{n}^{-1}}\}$
has a convergent subsequence $\{ \widehat{g_{n_{k}}^{-1}}\}$ to some $f\in E$. This makes $h$ a surjection, hence $h\in Iso(\widehat{X})$ and $h$ has an inverse in
$E$. To show the converse implication note that if $E$ is a group then $X_{l}=X_{p}$; cf. Proposition \ref{prop510}. Hence, by Theorem \ref{th512}, the map $\omega$
is proper.

To finish the proof of the proposition let us show that item (iii) implies item (ii) and vice versa. Note that $Iso(X)$ has a Weil completion if and only if the map
with $g\mapsto g^{-1}$ for every $g\in Iso(X)$ maps Cauchy sequences of $Iso(X)$ to Cauchy sequences; cf. \cite[Ch. III, \S 3.4 Theorem 1]{bour1}. It is easy to check
that in case when $Iso(X)$ is Cauchy-indivisible this is equivalent to $X_{l}=X_{p}$. Equivalently, by Proposition \ref{prop510},  $E$ is a group.
\end{proof}

\begin{remark} \label{rem517}
In case when $Iso(X)$ is a locally compact group, e.g. if $X$ is a locally compact space and $Iso(X)$ acts properly on it (as it is known in the case $X$ is
connected), then by \cite[Ch. III, \S 3 Exercise 8, p. 307]{bour1}, $Iso(X)$ has a locally compact completion hence $E$ is a locally compact group.
\end{remark}

We summarize with the following

\begin{corollary} \label{cor518}
If $E$ is a group the action $(Iso(X),X)$ is embedded densely in the proper action $(E,X\cup X_{l})$ such that the following equivariant diagram commutes
\[
\xymatrix{
(Iso(X),X)       \ar[d] \ar[r]   &X \ar[d] \\
(E,X\cup X_{l})  \ar[r]        &\widehat{X}}
\]
where $X\to X\cup X_{l}$ is the inclusion map and the map $Iso(X)\to E$ is defined by $g\mapsto \widehat{g}$ for every $g\in Iso(X)$. By the word ``densely" we mean
that $X$ is dense in $X\cup X_{l}$ and $\widehat{Iso(X)}$ is dense in $E$.
\end{corollary}

\begin{question} \label{que519}
The above embedding of a Cauchy-indivisible action as a dense sub-action of a proper one establishes a remarkable connection between Cauchy-indivisible and proper
actions. And at the same time proposes an interesting question: Is there any analogy with the situation of embedding of a proper action (on a locally compact and
connected space) in an appropriate zero-dimensional compactification, like in \cite{abels1} and \cite{MS2}? Namely, can we obtain any structurally informative
correspondence between divergent nets in $Iso(X)$ and suitable subsets of $X_l$?
\end{question}

\begin{remark} \label{rem519}
As we will see in the example described in \S 6 it may be happen that $X_{p}\neq X_{l}$ and $X\cup X_{l}=\widehat{X}$; cf. Proposition \ref{prop68}.
\end{remark}

In view of possible questions for refinements of Corollary \ref{cor518} we note that it may happen that $X\cup X_{p}=\widehat{X}$ and $E$ is not dense in
$Iso(\widehat{X})$, as the following example shows:

\begin{example} \label{ex520}
There is a separable metric space $(X,d)$ such that $(Iso(X),X)$ is Cauchy-indivisible, proper, $X\cup X_{p}=\widehat{X}$ and $Iso(X)$ has a Weil completion which
does not coincide with the group $Iso(\widehat{X})$.
\end{example}
\begin{proof}
Let $X$ be the set $\mathbb{Q}+\sqrt{2}\,\mathbb{N}$ endowed with the Euclidean metric; cf. Example \ref{ex43}. It is easy to check that $X\cup X_{p}=X\cup
X_{l}=\mathbb{R}$, see also Example \ref{ex515}, hence by Propositions \ref{prop510} and \ref{prop516}, $Iso(X)$ has a Weil completion (or just observe that $Iso(X)$
is an abelian group and use \cite[Ch. III, \S 3.5 Theorem 2]{bour1}). But all the reflections of the space are excluded, hence the pointwise closure $E$ of
$\widehat{Iso(X)}$ does not coincide with $Iso(\mathbb{R})$.
\end{proof}

\section{An example of a proper Cauchy-indivisible action of a group which has no Weil completion}
In this section we show that there is a separable metric space $X$ such that the action $(Iso(X),X)$ is proper, Cauchy-indivisible and the Ellis semigroup $E$ is not
a group. Equivalently, in view of Proposition \ref{prop516}, $Iso(X)$ has no Weil completion. Consider the space of the integers $\mathbb{Z}$ with the discrete metric
$d$, that is if $m,n\in\mathbb{Z}$ then $d(m,n)=0$ if $m=n$ and $d(m,n)=1$ otherwise. The group of isometries $Iso(\mathbb{Z})$ consists of all the self bijections of
$\mathbb{Z}$ and is an example of a topological group that has no Weil completion. To see that take $f_{n}:\mathbb{Z}\to\mathbb{Z}$ with $f_{n}z=z$ for $-n<z<0$,
$f_{n}(-n)=0$ and $f_{n}z=z+1$ otherwise. Then it is easy to verify that $f_{n}\to f$, where $fz=z$ for $z<0$, and $fz=z+1$ for $z\geq 0$. Hence  $\{ f_{n}z\}$ is a
Cauchy sequence in $\mathbb{Z}$ for every $z\in\mathbb{Z}$, therefore $\{ f_{n}\}$ is a Cauchy sequence in $Iso(\mathbb{Z})$ with respect to the uniformity of
pointwise convergence on $\mathbb{Z}$. But $\{f_{n}^{-1}0\}=\{ -n\}$ is not a Cauchy sequence, thus $\{ f_{n}^{-1}\}$ too. So, by \cite[Ch. III, \S 3.4 Theorem
1]{bour1}, $Iso(\mathbb{Z})$ has no Weil completion. The problem is that the action $(Iso(\mathbb{Z}),\mathbb{Z})$ is not Cauchy-indivisible. To see that notice that
$\{f_{n}^{-1}1\}=\{ 0\}$ but $\{f_{n}^{-1}0\}=\{ -n\}$ is not a Cauchy sequence. Nevertheless, the group  $Iso(Iso(\mathbb{Z}))$ is Cauchy-indivisible and has no Weil
completion as we show in the following.

Take an enumeration $A=\{ z_{i}\}$ of $\mathbb{Z}$ and  equip $Iso(\mathbb{Z})$ with the metric
\[
\varrho(f,g)=\sum_{i=1}^{\infty} \frac{1}{3^{i}}\,\, d(fz_{i},gz_{i})
\]
for $f,g\in Iso(\mathbb{Z})$. In view of Proposition \ref{prop53} the uniformity of pointwise convergence, the left uniformity and the uniformity induced by $\varrho$
on $Iso(\mathbb{Z})$ coincide (the choice of $\frac{1}{3}$ instead of $\frac{1}{2}$ in Definition \ref{def52} will be clarified in the proof of Lemma \ref{lem61}).
Note that $(Iso(\mathbb{Z}),\varrho)$ is a separable metric space. We will show that $Iso(Iso(\mathbb{Z}))$ is Cauchy-indivisible but has no Weil completion.

\begin{lemma} \label{lem61}
If $T\in Iso(Iso(\mathbb{Z}))$ and $f,g\in Iso(\mathbb{Z})$ then
\[
d(T(f)z,T(g)z)=d(fz,gz)
\]
for every $z\in\mathbb{Z}$.
\end{lemma}
\begin{proof}
Since $\varrho(T(f),T(g))=\varrho(f,g)$ then
\[
\sum_{i=1}^{\infty} \frac{1}{3^{i}}\,\, d(T(f)z_{i},T(g)z_{i})=\sum_{i=1}^{\infty} \frac{1}{3^{i}}\,\, d(fz_{i},gz_{i}).
\]
Since the values of $d$ are $0$ or $1$ then $d(T(f)z_{n},T(g)z_{n})=d(fz_{n},gz_{n})$, for every $z_{n}\in A=\mathbb{Z}$ (here is the role of the choice of
$\frac{1}{3}$ instead of $\frac{1}{2}$).
\end{proof}

\begin{proposition} \label{prop62}
If $T\in Iso(Iso(\mathbb{Z}))$ and $f\in Iso(\mathbb{Z})$ then $T(f)=T(e)\circ f$, where $e$ is the unit element of $Iso(\mathbb{Z})$.
\end{proposition}
\begin{proof}
Let $z_{k},z_{l}$ be two distinct integers and let $g\in Iso(\mathbb{Z})$ be such that $gz_{k}=z_{l}$, $gz_{l}=z_{k}$ and $gz=z$ elsewhere. We show that
$T(g)=T(e)\circ g$. If $z\neq z_{k},z_{l}$ then, by Lemma \ref{lem61}, $d(T(g)z,T(e)z)=d(gz,z)=0$. Hence $T(g)z=T(e)z=T(e)\circ gz$. Moreover
\[
d(T(g)z_{k},T(e)z_{k})=d(gz_{k},z_{k})=d(z_{l},z_{k})=1
\]
and, similarly, $d(T(g)z_{l},T(e)z_{l})=1$. Since $T(g)z_{k}\neq T(g)z=T(e)z$ for $z\neq z_{k},z_{l}$ and $T(e)$ is surjective then $T(g)z_{k}=T(e)z_{l}=T(e)\circ
gz_{k}$ and, similarly, $T(g)z_{l}=T(e)\circ gz_{l}$. Therefore $T(g)=T(e)\circ g$.

Fix $f\in Iso(\mathbb{Z})$ and some $z\in\mathbb{Z}$. If $fz=z$ then $T(f)z=T(e)z=T(e)\circ fz$ since $d(T(f)z,T(e)z)=d(fz,z)=0$. If $fz\neq z$, let $g\in
Iso(\mathbb{Z})$ with $gz=fz$, $gfz=z$ and $gw=w$ elsewhere. Since $d(T(f)z,T(g)z)=d(fz,gz)=0$ then $T(f)z=T(g)z$. Using the result of the previous paragraph,
$T(f)z=T(g)z=T(e)\circ gz=T(e)\circ fz$. Since $z$ was arbitrary then $T(f)=T(e)\circ f$.
\end{proof}

\begin{corollary} \label{cor63}
Let $L,T\in Iso(Iso(\mathbb{Z}))$. Then $L\circ T(e)=L(e)\circ T(e)$ and $T^{-1}(e)=(T(e))^{-1}$.
\end{corollary}
\begin{proof}
Since $T(f)=T(e)\circ f$ for every $T\in Iso(Iso(\mathbb{Z}))$ and $f\in Iso(\mathbb{Z})$, then
\[
L\circ T(f)=L(T(f))=L(e)\circ T(f)=L(e)\circ T(e)\circ f.
\]
Hence, $L(e)\circ T(e)= L\circ T(e)$. If $I$ denote the identity on $Iso(Iso(\mathbb{Z}))$, then $f=I(f)=I(e)\circ f$. Hence $I(e)=e$ and $T^{-1}(e)=(T(e))^{-1}$.
\end{proof}

\begin{proposition} \label{prop64}
The map $\mathcal{B}:Iso(Iso(\mathbb{Z}))\to Iso(\mathbb{Z})$ with $\mathcal{B}(T)=T(e)$ is a uniform group isomorphism with respect to the uniformities of pointwise
convergence on the underlying spaces $Iso(\mathbb{Z})$ and $\mathbb{Z}$ respectively.
\end{proposition}
\begin{proof}
By Proposition \ref{prop53} we can equip $Iso(Iso(\mathbb{Z}))$ with a left invariant metric $\sigma$ such that the uniformity of pointwise convergence, the left
uniformity and the uniformity induced by $\sigma$ on $Iso(Iso(\mathbb{Z}))$ coincide. Let $L_{n},T_{n}\in Iso(Iso(\mathbb{Z}))$ such that $\sigma(L_{n},T_{n})\to 0$,
hence $\sigma(T_{n}^{-1}L_{n},I)\to 0$. Therefore $T_{n}^{-1}L_{n}\to I$ pointwise on $Iso(\mathbb{Z})$ so $T_{n}^{-1}L_{n}(e)\to e$, thus
$\varrho(L_{n}(e),T_{n}(e))\to 0$. For the converse, note that if $\varrho(T_{n}^{-1}L_{n}(e),e)\to 0$ then $\varrho(T_{n}^{-1}L_{n}(e)\circ f,f)\to 0$ for every
$f\in Iso(\mathbb{Z})$ since the map $Iso(\mathbb{Z})\to Iso(\mathbb{Z})$ with $g\mapsto gf$ is continuous. Hence $T_{n}^{-1}L_{n}\to I$ pointwise on
$Iso(\mathbb{Z})$. Corollary \ref{cor63} implies that $\mathcal{B}$ is also group isomorphism.
\end{proof}

\begin{proposition} \label{prop65}
The group $Iso(Iso(\mathbb{Z}))$ is Cauchy-indivisible and has no Weil completion.
\end{proposition}
\begin{proof}
Let us show firstly that $Iso(Iso(\mathbb{Z}))$ is Cauchy-indivisible. Let $\{ T_{n}\}\subset Iso(Iso(\mathbb{Z}))$ and $f\in Iso(\mathbb{Z})$ such that $\{
T_{n}(f)\}$ is a Cauchy sequence in $Iso(\mathbb{Z})$. Take some $g\in Iso(\mathbb{Z})$. Since $\{ T_{n}(f)\}$ is a Cauchy sequence in $Iso(\mathbb{Z})$ then, it is
easy to see that$\{ T_{n}(f)z\}$ is a Cauchy sequence for every $z\in\mathbb{Z}$. Equivalently, $\{ T_{n}(f)f^{-1}gz\}$ is a Cauchy sequence for every
$z\in\mathbb{Z}$. By Proposition \ref{prop62},
\[
T_{n}(f)f^{-1}gz=T_{n}(e)\circ ff^{-1}gz=T_{n}(e)\circ gz=T_{n}(g)z.
\]
Therefore $\{ T_{n}(g)\}$ is a Cauchy sequence in $Iso(Iso(\mathbb{Z}))$ for every $g\in Iso(\mathbb{Z})$, hence $Iso(Iso(\mathbb{Z}))$ is Cauchy-indivisible.

Since by the previous proposition the groups $Iso(Iso(\mathbb{Z}))$ and $Iso(\mathbb{Z})$ are uniformly isomorphic and the group $Iso(\mathbb{Z})$ has no Weil
completion then the same also holds for $Iso(Iso(\mathbb{Z}))$.
\end{proof}

\begin{proposition} \label{prop66}
The action $(Iso(Iso(\mathbb{Z})),Iso(\mathbb{Z}))$ is proper.
\end{proposition}
\begin{proof}
Let $f,g\in Iso(\mathbb{Z})$ and $\{ T_{n}\}\subset Iso(Iso(\mathbb{Z}))$ be a sequence such that $T_{n}(f)\to g$. Hence, by Proposition \ref{prop62}, $T_{n}(e)\circ
f\to g$ thus $T_{n}(e)\to gf^{-1}$. Therefore $\{ T_{n}(h)\}$ converges for every $h\in Iso(\mathbb{Z})$. Since $(T_{n}(e))^{-1}\to fg^{-1}$ it is easy to verify that
$\{ T_{n}\}$ converges in $Iso(Iso(\mathbb{Z}))$ hence the action $(Iso(Iso(\mathbb{Z})),Iso(\mathbb{Z}))$ is proper.
\end{proof}

\begin{remark} \label{rem67}
Notice that  $Iso(Iso(\mathbb{Z}))$ is not locally compact since has no Weil completion ($Iso(\mathbb{Z})$ is, of course, not locally compact).
\end{remark}

In the following proposition we give a precise description of the sets $\widehat{X}$, $X_p$, $X_l$ and $E$ for the action $(Iso(Iso(\mathbb{Z})),Iso(\mathbb{Z}))$.

\begin{proposition} \label{prop68}
The following holds.
\begin{enumerate}
\item[(i)] The pointwise closure of $Iso(\mathbb{Z})$ in the set of all selfmaps of $\mathbb{Z}$  consists of all the injective selfmaps of $\mathbb{Z}$ and coincides
with the completion of $(Iso(\mathbb{Z}),\varrho)$.

\item[(ii)] The set $X_{p}$ is empty and the set $X_l$ consists of all the injective selfmaps of  $\mathbb{Z}$ which are not surjective. Moreover, the completion of
$(Iso(\mathbb{Z}),\varrho)$ coincides with the set $X\cup X_{l}$.

\item[(iii)] The Ellis semigroup $E$ is uniformly homeomorphic to  the semigroup of all injective selfmaps of $\mathbb{Z}$.

\item[(iv)] For every non-surjective map $f$ in the completion of $(Iso(\mathbb{Z}),\varrho)$ (that is $f$ does not belong to $Iso(\mathbb{Z})$) there exists
a sequence $\{T_n\}\subset E$ with $T_n\to\infty$ in $E$ and $T_n(f)\to f$. Hence, $E$ does not act properly on the completion of $(Iso(\mathbb{Z}),\varrho)$.
\end{enumerate}
\end{proposition}
\begin{proof}
(i) Consider an injective map $f:\mathbb{Z}\to\mathbb{Z}$. There is a strictly increasing sequence of sets $A_n\subset\mathbb{Z}$  (with respect to the inclusion)
such that each $A_n$ contains finitely many points and contains also the sets $[-n,n]$ and $f([-n,n])$, for every $n\in\mathbb{N}$. Define $g_n\in Iso(\mathbb{Z})$
such that the restriction of $g_n$ on $A_n$ is a permutation which coincides with $f$ on the interval $[-n,n]$ and $g_n$ is the identity on the complement of $A_n$.
Hence, $g_n\to f$ pointwise.

(ii) Assume that $X_p\neq\emptyset$. Then there exist a sequence $\{ T_{n}\}\subset Iso(Iso(\mathbb{Z}))$ and a map $f\in Iso(\mathbb{Z})$ such that $\{ T_{n}(f)\}$
and $\{ T_{n}^{-1}(f)\}$ are Cauchy sequences in $Iso(\mathbb{Z})$ and $T_n\to\infty$. The Cauchy-indivisibility of the action
$(Iso(Iso(\mathbb{Z})),Iso(\mathbb{Z}))$ (cf. Proposition \ref{prop65}) implies that $\{ T_{n}(e)\}$ and  $\{ T_{n}^{-1}(e)\}$ are Cauchy sequences in
$Iso(\mathbb{Z})$, where $e$ denote the unit element of $Iso(\mathbb{Z})$. Hence, for every $z\in\mathbb{Z}$ the sequences $\{ T_n(e)(z)\}$ and $\{
T_{n}^{-1}(e)(z)=(T_{n}(e))^{-1}(z)\}$ are Cauchy in $\mathbb{Z}$ (cf. Corollary \ref{cor63}). Thus, for every $z\in\mathbb{Z}$ the sequences $\{ T_n(e)(z)\}$ and $\{
(T_{n}(e))^{-1}(z)\}$ are eventually constant. So, there exist injective selfmaps $g,h$ of $\mathbb{Z}$ such that $T_n(e)\to g$ and $(T_{n}(e))^{-1}\to h$. Obviously,
$h=g^{-1}$. This shows that $g\in Iso(\mathbb{Z})$.  The properness of the action $(Iso(Iso(\mathbb{Z})),Iso(\mathbb{Z}))$ (cf. Proposition \ref{prop66}) implies that
the sequence $\{T_n\}$ has a convergent subsequence in $Iso(Iso(\mathbb{Z}))$ (actually, by the proof of Proposition \ref{prop66}, the whole sequence $\{T_n\}$
converges) which is a contradiction. Therefore $X_p=\emptyset$.

By item (i) it follows that $X_l$ is contained in the semigroup of all the injective selfmaps of $\mathbb{Z}$ which are not surjective and by the same proof it
follows that if $f$ is a selfmap of $\mathbb{Z}$ which is not surjective then there exists a sequence $\{g_n\}\subset Iso(\mathbb{Z})$ such that $g_n\to f$ pointwise
on $\mathbb{Z}$. Set $T_{n}\in Iso(Iso(\mathbb{Z}))$ with $T_{n}(e)=g_{n}$ (that is $T_n(h)=T_n(e)\circ h=g_nh$ for every $h\in Iso(\mathbb{Z})$). Then
$T_n\to\infty$, $\{T_n(e)\}$ is a Cauchy sequence in $Iso(\mathbb{Z})$ and $T_n(e)\to f$, hence $f\in X_l$.

By item (i)  a map in the completion of $(Iso(\mathbb{Z}),\varrho)$ is an injection $\mathbb{Z}\to\mathbb{Z}$. If this map is also a surjection  then it belongs to
$X=Iso(\mathbb{Z})$ and if it is not a surjection then it belongs to $X_l$. Hence, the completion of $(Iso(\mathbb{Z}),\varrho)$ coincides with the set $X\cup X_{l}$.

(iii) It follows easily by Proposition \ref{prop64}.

(iv) Let $f$ be in the completion of $(Iso(\mathbb{Z}),\varrho)$ but does not belong to $Iso(\mathbb{Z})$. Hence, $f$ is not surjective. Take a point $z\notin
f(\mathbb{Z})$. There is a strictly increasing sequence of sets $A_n\subset\mathbb{Z}$  such that each $A_n$ contains finitely many points and contains strictly the
sets $\{ z\}$, $[-n,n]$ and $f([-n,n])$, for every $n\in\mathbb{N}$. Define $g_n\in Iso(\mathbb{Z})$ such that the restriction of $g_n$ on $A_n$ is a permutation with
the property $g_n$ is the identity on the set $f([-n,n])$, $g_nz>n$ and $g_n$ is the identity on the complement of $A_n$. Hence, $g_n\circ f\to f$ pointwise and since
$f$ is not surjective then $g_n\to\infty$ in $Iso(\mathbb{Z})$. Set $T_{n}\in Iso(Iso(\mathbb{Z}))$, as above, with $T_{n}(e)=g_{n}$. Then $T_n\to\infty$ in $E$
because $T_n(e)=g_n\to\infty$ in $Iso(\mathbb{Z})$ and $T_n(f)=g_n\circ f \to f$.
\end{proof}

\section{Sections, Borel Sections, Fundamental Sets and Cauchy-indivisibility}
As it is indicated in the introduction a section of an action $(G,X)$ is a subset of $X$ which contains only one point from each orbit. If a section is a Borel subset
of $X$ it called a Borel section. Concerning the existence of Borel sections, if $(Y, d)$ is a separable metric space and $\mathcal{R}$ is an equivalence relation on
$Y$ such that the $\mathcal{R}$-saturation of each open set is Borel, then there is a Borel set $S$ whose intersection with each $\mathcal{R}$-equivalence class which
is complete with respect to $d$ is nonempty, and whose intersection with each $\mathcal{R}$-equivalence class is at most one point; cf. \cite[Lemma 2]{kallman}. The
problem of the existence of a Borel section for a continuous Polish action is of remarkable significance because the existence of a Borel section is equivalent to
many interesting facts, like that the underlying space has only trivial ergodic measures, the orbit space has a standard Borel structure and it has no non-trivial
atoms. Recall that an action $(G,X)$ is called \textit{Polish} if both $G$ and $X$ are Polish spaces, i.e. they are separable and metrizable by a complete metric.
Keeping the previous in mind we have the following:

\begin{proposition} \label{prop71}
If the Ellis semigroup $E$ is a group then the action $(E,X\cup X_{l})$ has a Borel section.
\end{proposition}
\begin{proof}
Assume that the Ellis semigroup $E$ is a group. Since by Proposition \ref{prop510} we have $X_l=X_p$ and by Proposition \ref{prop516} the map $\omega :E\times (X\cup
X_{l})\to (X\cup X_{l})\times \widetilde{X}$ is proper then each orbit $Ex$, $x\in X\cup X_{l}$ is closed in $\widetilde{X}$. Hence, by \cite[Lemma 2]{kallman} there
exists a Borel set $S\subset \widetilde{X}$ such that $S\cap (X\cup X_{l})$ is a Borel section (with respect to the relative topology of $X\cup X_{l}$) for the action
$(E,X\cup X_{l})$.
\end{proof}

A very useful notion in the theory of proper actions on locally compact spaces with paracompact orbit space is the notion of a fundamental set.

Let $G$ be a topological group which acts continuously on a topological space $X$ and $A,B\subset X$. Let us call the set $G_{AB}:=\{g\in G\,:\, gA\cap B\neq\emptyset
\}$ the \textit{transporter} from $A$ to $B$.

\begin{definition} \label{def74}
A subset $F$ of $X$ is called a fundamental set for the action  $(G,X)$ if the following holds.
\begin{itemize}
\item[(i)] $GF=X$.
\item[(ii)] For every $x\in X$ there exists a neighborhood $V\subset X$ of $x$ such that the transporter $G_{VF}$ of $V$ to $F$ has compact closure in $G$.
\end{itemize}
\end{definition}
For locally compact spaces we can replace condition (ii) with the following equivalent condition.
\medskip

(iia)  The transporter $G_{KF}$ from $K$ to $F$ has compact closure in $G$ for every non-empty compact subset $K$ of $X$.

\medskip

Note that the existence of a fundamental set implies that the action group $G$ is locally compact and the action $(G,X)$ is proper.

The notion of a fundamental set is relative to the notion of a section but it is different in general, in the sense that there are cases where a section is a
fundamental set, a fundamental set fails to be a section and cases where a section fails to be a fundamental set. A section may not be Borel or even if it is Borel
may not to be contained in any fundamental set as the following example shows.

\begin{example} \label{ex76}
The action $(\mathbb{Z},\mathbb{R})$ with $(z,r)\mapsto r+z$, $z\in\mathbb{Z}, r\in\mathbb{R}$ is proper and it has a Borel section which is not contained in any
fundamental set. Indeed, it is easy to see that the set $S:=([0,1)\setminus \bigcup_{n\in\mathbb{N}} \{ \frac{1}{n}\})\cup \bigcup_{n\in\mathbb{N}} \{ n+\frac{1}{n}
\}$ is a section because the interval $[0,1)$ is a section (and a fundamental set) for the action $(\mathbb{Z},\mathbb{R})$. Take an open ball $B$ centered at $0$
with radius $\varepsilon>0$. Then there exists $n_0\in\mathbb{N}$ such that $\frac{1}{n}<\varepsilon$ for every $n\geq n_0$. Let $A$ be a  subset of $\mathbb{R}$ that
contains $S$. Hence $\{ n\,|\, n\geq n_0 \}$ is a subset of the transporter $\mathbb{Z}_{BS}\subset \mathbb{Z}_{BA}$, so $A$ can not be a fundamental set.

It is also possible to construct a section which is not Borel. Take a set $D\subset [0,1)$ which is not a Borel set and consider the set $S_1:=D\cup \{ x+2\,
|\,x\in\mathbb{R}\setminus D\}$. Obviously $S_1$ is a section which is not a Borel subset of the reals.
\end{example}

Nevertheless sections, Borel section and fundamental sets have a very strong connection as the following theorem shows.

\begin{theorem} \label{sectfund}
Let $G$ be a locally compact group which acts properly on a locally compact space $X$, and suppose that the orbit space $G\backslash X$ is paracompact. Let $S$ be a
section for the action $(G,X)$. Then
\begin{itemize}
\item[(i)] For every open neighborhood $U$ of $S$ we can construct a closed fundamental set $F_c$ and an open fundamental set $F_o$ such that $F_c\subset F_o\subset U$.
\item[(ii)] If, in addition, $(X,d)$ is a separable metric space, in which case, by Theorem \ref{th33} the action $(G,X)$ is Cauchy-indivisible,  then there exists
a Borel section $S_B$, which is also a fundamental set, such that $S_B\subset F_c\subset F_o\subset U$.
\end{itemize}
\end{theorem}
\begin{proof}
(i) Since $U$ is open it is a union of open balls, let say  $S_i$, $i\in I$. Let $p: X\to G\backslash X$ be the natural map $x\mapsto Gx$. Then $p(S_i)$, $i\in I$ is
an open covering of the locally compact and paracompact space $ G\backslash X$. Hence, there is a locally finite refinement $\{ W_j\}$, $j\in J$ which consists of
open subsets of $ G\backslash X$ with compact closures such that $W_j\subset p(S_{i_j})$, for some $i_j\in I$. Now we can follow the classical proof for the existence
of fundamental sets; cf. \cite[Lemma 2, p. 8]{koszul}. Let $\{ V_j\}$ be an open covering of $ G\backslash X$ such that $\overline{V_j}\subset W_j$ for every $j\in
J$. Fix an index $j\in J$ and consider the restriction of the natural map $p:X\to G\backslash X$ on the open ball $S_{i_j}$. Since $S_{i_j}$ is locally compact then
there exist an open set $U_{i_j}\subset S_{i_j}$ with compact closure and a compact set $K_{i_j}\subset U_{i_j}\subset S_{i_j}$ such that $p(U_{i_j})=W_j$ and
$p(K_{i_j})=\overline{V_j}$. Let $F_c:=\bigcup_j K_{i_j}$ and $F_o :=\bigcup_j U_{i_j}$. The family $\{ U_{i_j}\}_{j\in J}$ is locally finite in $X$ hence the set
$F_c$ is closed; cf. \cite[Ch. I, \S 1.5 Proposition 4]{bour1}. Moreover, $GF_c=X$. Take a point $x\in X$ and neighborhood $A$ of $x$ with compact closure. Since the
covering $\{ W_j\}_{j\in J}$ is locally finite, then the transporters $G_{AU_{i_j}}$ from $A$ to $U_{i_j}$ are non-empty for only finitely many $j\in J$. Since the
sets $A$ and $U_{i_j}$ have compact closure and the action $(G,X)$ is proper, then the transporter $G_{AF_o}$ of $A$ to $F_o$ has compact closure in $G$. Thus, $F_c$
and $F_o$ are fundamental sets and by construction  $F_c\subset F_o\subset U$.

(ii) Let $F_c$ a closed fundamental set for the action $(G,X)$ like in item (i). Define a relation $\mathcal{R}$ on $F_c$ with  $x\mathcal{R} y$, $x,y\in F_c$ if and
only if $y\in Gx$. We will find a Borel section for the closed fundamental set $F_c$ with respect to the previous natural relation on $F_c$ and then we will show that
it is, also, a Borel section for the action $(G,X)$. Obviously $\mathcal{R}$ is an equivalence relation on the separable metric space $(F_c,d)$. Since the action
$(G,X)$ is proper each orbit $Gx$ is closed in $X$, for every $x\in X$. The $\mathcal{R}$-equivalence class of a point $x\in F_c$ is $Gx\cap F_c$, hence it is a
closed subset of $X$, thus it is complete space with respect to the metric $d$. If $U$ is an open subset of $F_c$ with respect to the relative topology of $F_c$ then
the $\mathcal{R}$-saturation of $U$ is the set $GU\cap F_c$ which is open in $F_c$ hence it is a Borel set. Therefore we can apply \cite[Lemma 2]{kallman} to find a
Borel section $S_B\subset F_c$ for the equivalence relation $\mathcal{R}$. Moreover, $S_B$ is a Borel section (and a fundamental set) for the action $(G,X)$, since it
is contained in the closed fundamental set $F_c$.
\end{proof}

\begin{remark} \label{rem78}
Note that the assumption that the orbit space $G\backslash X$ is paracompact is automatically satisfied for proper isometric actions. Actually they are metrizable by
the metric defined by
\[
\rho (Gx,Gy):=\inf \{d(gx,hy)\,|\,g,h\in G\}=\inf \{d(x,hy)\,|\,h\in G\}
\]
where $G=Iso(X)$ or $E$. So we can apply Theorem \ref{sectfund} in both cases.
\end{remark}

\begin{question} \label{que73}
As Theorem \ref{sectfund}(ii) indicates the notion of a Borel section is remarkably related to that of a fundamental set in the locally compact case and may be,
similarly, used for structural theorems. Note that the Borel section $S_B$, because of its construction, is a ``minimal" fundamental set for the action $(G,X)$, that
is for each point $x\in X$ the transporter $G_{\{ x\}S_B}=gG_x$ for some $g\in G$. So, it is interesting to ask whether the existing Borel section for the action
$(E,X\cup X_l)$ can be reduced to a Borel section for the initial action $(Iso(X),X)$.
\end{question}

\end{document}